\documentclass[12pt,twoside]{amsart}
\usepackage{amsmath}
\usepackage{amsthm}
\usepackage{amsfonts}
\usepackage{amssymb}
\usepackage{latexsym}
\usepackage{mathrsfs}
\usepackage{amsmath}
\usepackage{amsthm}
\usepackage{amsfonts}
\usepackage{amssymb}
\usepackage{latexsym}
\usepackage{geometry}
\usepackage{dsfont}
\usepackage[dvips]{graphicx}
\usepackage{color}
\usepackage[all]{xy}

\date{}
\allowdisplaybreaks[4] \footskip=15pt
\renewcommand{\uppercasenonmath}[1]{}

\usepackage{graphicx,amssymb}
\usepackage[all]{xy}
\usepackage{amsmath}

\allowdisplaybreaks
\usepackage{amsthm}
\usepackage{color}

\theoremstyle{plain}
\newtheorem{theorem}{Theorem}[section]

\newtheorem{lemma}[theorem]{Lemma}
\newtheorem{corollary}[theorem]{Corollary}

\newtheorem*{open question}{Open Question}

\theoremstyle{definition}
\newtheorem*{acknowledgement}{Acknowledgement}
\newtheorem*{theo}{Theorem}

\theoremstyle{remark}

\newcommand{\N}{\mathcal{N}}

\def\p{\frak p}
\def\m{\frak m}

\def\GV{{\rm GV}}
\def\tor{{\rm tor_{\rm GV}}}
\def\Hom{{\rm Hom}}
\def\Ext{{\rm Ext}}

\def\Ker{{\rm Ker}}

\def\Coker{{\rm Coker}}

\def\Ann{{\rm Ann}}

\def\GV{{\rm GV}}
\def\Max{{\rm Max}}
\def\DW{{\rm DW}}

\def\Max{{\rm Max}}

\begin{document}
\begin{center}
{\large  \bf On two versions of Cohen's theorem for modules}

\vspace{0.5cm}   Xiaolei Zhang$^{a}$,\  Hwankoo Kim$^{b}$,\ Wei Qi$^{c}$

{\footnotesize a.\ School of Mathematics and Statistics, Shandong University of Technology, Zibo 255049, China\\
b.\ Division of Computer Engineering, Hoseo University, Asan 31499, Republic of Korea\\
c.\  School of Mathematical Sciences, Sichuan Normal University, Chengdu 610068,  China\\

E-mail: hkkim@hoseo.edu\\}
\end{center}

\bigskip
\centerline { \bf  Abstract}
\bigskip
\leftskip10truemm \rightskip10truemm \noindent

Parkash and  Kour obtained  a new version of Cohen's theorem for Noetherian modules, which states that  a finitely generated $R$-module $M$ is Noetherian if and only if for every prime ideal $\p$ of $R$ with $\Ann(M)\subseteq \p$, there exists a finitely generated  submodule $N^\p$ of $M$ such that $\p M\subseteq N^\p\subseteq M(\p)$, where $M(\p)=\{x\in M\mid sx\in \p M $ for some $s\in R \setminus \p \}$.
In this paper, we generalize the Parkash and  Kour version of Cohen's theorem for Noetherian modules to those for $S$-Noetherian modules and $w$-Noetherian modules.
\vbox to 0.3cm{}

\leftskip0truemm \rightskip0truemm
\bigskip

\section{Introduction}
Throughout this article, all rings are commutative rings with identity and all modules are unitary. Let $R$ be a ring and $M$ an $R$-module. For a subset $U$ of $M$, we denote by $\langle U\rangle$ the submodule of $M$ generated by $U$. Early in 1950, Cohen showed that a ring $R$ is Noetherian if and only if every prime ideal of $R$ is finitely generated \cite[Theorem 2]{c50}. Let $\p$ be a prime ideal of $R$. Following \cite{mccs93}, we set $M(\p):=\{x\in M \mid sx\in \p M $ for some $s\in R \setminus \p \}$. Then $M(\p)$ is obviously a submodule of $M$. In 1994, Smith  extended  Cohen's Theorem from rings to  modules, which states that a finitely generated $R$-module $M$ is Noetherian if and only if the submodules $\p M$ of $M$ are finitely generated for every prime ideal $\p$ of $R$, if and only if $M(\p)$ is finitely generated  for each prime ideal $\p$ of $R$ with $\p\supseteq \Ann(M)$ \cite{s94}. Recently,    Parkash and  Kour generalized the Smith's result on Noetherian modules as follows:

\begin{theo}\cite[Theorem 2.1]{pk21} Let $R$ be a ring and  $M$  a finitely generated $R$-module. Then $M$ is Noetherian if and only if for every prime ideal $\p$ of $R$ with $\Ann(M)\subseteq \p$, there exists a finitely generated  submodule $N^\p$ of $M$ such that $\p M\subseteq N^\p\subseteq M(\p)$.
\end{theo}

In the past few decades, some generalizations of Noetherian rings or Noetherian modules have been extensively studied, especially via some multiplicative subsets $S$ of $R$ and the $w$-operation (see \cite{ad02,l15,l07,kkl14,ywzc11,zwk} for example).  And the related  Cohen's theorem has also been considered by many authors. In 2002, Anderson and Dumitrescu gave an analogue of Cohen's theorem for $S$-Noetherian modules, which states that an $S$-finite module $M$ is $S$-Noetherian if and only if the submodules of the form $\p M$ are $S$-finite for each prime ideal $\p$ of $R$ (disjoint from $S$) \cite[Proposition 4]{ad02}. In 1997, Wang and McCasland obtained an analogue of Cohen's theorem for strong Mori (SM) modules $M$ over integer domains for which  $M$  satisfies the ascending chain condition on $w$-submodules of $M$. In fact, they showed that a $w$-module $M$ is an SM module if and only if each $w$-submodule of $M$ is $w$-finite type, if and only if $M$ and every prime $w$-submodule of $M$ are $w$-finite type \cite[Theorem 4.4]{fm97}. In this paper,  we give both an $S$-analogue and a $w$-analogue of the Parkash and  Kour's result on Noetherian modules, which can be seen as generalizations of Cohen's theorem for modules.

\section{The Cohen's theorem for $S$-Noetherian modules}
Let $R$ be a ring and $S$ a multiplicative subset of $R$, that is  $1\in S$ and $s_1s_2\in S$ for any $s_1\in S$, $s_2\in S$. Let $M$ be an $R$-module. Recall from  \cite{ad02} that  $M$ is called \emph{$S$-finite} if $sM\subseteq F$ for some $s\in S$ and some finitely generated submodule $F$ of $M$. Also, $M$ is called \emph{$S$-Noetherian} if each submodule of $M$ is an $S$-finite $R$-module. Then $R$ is called an $S$-Noetherian ring if $R$ is $S$-Noetherian as an $R$-module.  Anderson and Dumitrescu obtained a Cohen-type theorem for $S$-Noetherian modules: An $S$-finite $R$-module $M$ is $S$-Noetherian if and only if the submodules of the form $\p M$ are $S$-finite for each prime ideal $\p$ of $R$ (disjoint from $S$) \cite[Proposition 4]{ad02}. Now we give a ``stronger" version of Cohen's theorem for $S$-Noetherian modules which can be seen as an $S$-analogue of the Parkash and Kour's result \cite[Theorem 2.1]{pk21}.

\begin{theorem} Let $R$ be a ring and $S$ a multiplicative subset of $R$. Let $M$ be an $S$-finite $R$-module. Then $M$ is $S$-Noetherian if and only if for every prime ideal $\p$ of $R$ with $\Ann(M)\subseteq \p$, there exists an $S$-finite submodule $N^\p$ of $M$ such that $\p M\subseteq N^\p\subseteq M(\p)$.
\end{theorem}

\begin{proof} Suppose that $M$ is an $S$-Noetherian $R$-module and let $\p$ be a prime ideal with $\Ann(M)\subseteq \p$.  If we take $N^\p:=\p M$, then $N^\p$ is certainly an $S$-finite submodule of $M$ satisfying $\p M\subseteq N^\p\subseteq M(\p)$.

Conversely, suppose on the contrary that $M$ is not $S$-Noetherian. Let $\N$ be the set of all submodules of $M$ which are not $S$-finite. Then $\N$ is non-empty. Make a partial order on $\N$ by defining  $N_1\leq N_2$ if and only if $N_1\subseteq N_2$ in $\N$. Let $\{N_i\mid i\in \Lambda\}$  be a chain in  $\N$. Set $N:=\bigcup\limits_{i\in \Lambda}N_i$. Then $N$ is not $S$-finite. Indeed, suppose $sN\subseteq \langle x_1, \dots , x_n\rangle\subseteq N$ for some $s\in S$. Then there exists $i_0\in \Lambda$ such that $\{ x_1, \dots , x_n\}\subseteq N_{i_0}$. Thus $sN_{i_0}\subseteq sN\subseteq \langle x_1, \dots , x_n\rangle\subseteq N_{i_0}$ implies that $N_{i_0}$ is $S$-finite, which is a contradiction.
Then by Zorn's Lemma, $\N$ has a maximal element, which is also denoted by $N$. Set $\p:=(N:M)=\{r\in R \mid rM\subseteq N\}$.

We claim that $\p$ is a prime ideal of $R$. Assume on the contrary that there exist $a,b\in R \setminus \p$ such that $ab\in \p$. Since $a,b\in R \setminus \p$, we have $aM\not\subseteq N$ and $bM\not\subseteq N$. Therefore $N+aM$ is $S$-finite. Let $\{y_1,\dots,y_m\}$ be a subset of $N+aM$ such that $s_1(N+aM)\subseteq \langle y_1,\dots,y_m\rangle$ for some $s_1 \in S$. Write $y_i=w_i+az_i$ for some $w_i\in N$ and $z_i\in M\ (1\leq i\leq m)$. Set $L:=\{x\in M\mid ax\in N\}$. Then $N+bM\subseteq L$, and hence $L$ is also $S$-finite. Let $\{x_1,\dots,x_k\}$ be a subset of $L$ such that $s_2L \subseteq\langle x_1,\dots,x_k \rangle$ for some $s_2 \in S$. Let $n \in N$ and write $$s_1n=\sum\limits_{i=1}^mr_iy_i=\sum\limits_{i=1}^mr_iw_i+a\sum\limits_{i=1}^mr_iz_i.$$
Then $\sum\limits_{i=1}^mr_iz_i\in L$. Thus $s_2\sum\limits_{i=1}^mr_iz_i=\sum\limits_{i=1}^kr'_ix_i$ for some $r'_i\in R$ ($i=1,\dots,k$).
So $s_1s_2n=\sum\limits_{i=1}^ms_2r_iw_i+\sum\limits_{i=1}^kr'_iax_i$. Thus $s_1s_2N\subseteq \langle w_1,\dots, w_m, ax_1,\dots, ax_k \rangle$ implies that $N$ is $S$-finite, which is a  contradiction.

We also claim that $M(\p)\subseteq N$. Suppose on the contrary that there exists $y\in M(\p)$ such that $y\not\in N$. Then there exists $t\in R \setminus \p$ such that $ty\in \p M=(N:M)M\subseteq N$. As $t\not\in \p=(N:M)$,  it follows that $tM\not\subseteq N$. Therefore $N+tM$ is $S$-finite. Let $\{u_1,\dots,u_m\}$ be a subset of $N+tM$ such that $s_3(N+tM)\subseteq \langle u_1,\dots,u_m \rangle$ for some $s_3\in S$. Write $u_i=w_i+tz_i ~  (i=1,\dots,m)$ with $w_i\in N$ and $z_i\in M$. Set $T:=\{x\in M \mid tx\in N\}$. Then $N\subset N+Ry\subseteq T$, and hence $T$ is $S$-finite. Then there exists a subset $\{v_1,\dots,v_l\}$ of $T$ such that $s_4T\subseteq \langle v_1,\dots,v_l\rangle$ for some $s_4 \in S$. Let $n$ be an element in $N$. Then $$s_3n=\sum\limits_{i=1}^mr_iu_i=\sum\limits_{i=1}^mr_iw_i+t\sum\limits_{i=1}^mr_iz_i.$$ Thus $\sum\limits_{i=1}^mr_iz_i\in T$. So $s_4\sum\limits_{i=1}^mr_iz_i=\sum\limits_{i=1}^lr'_iv_i$ for some $r'_i\in R ~(i=1,\dots,l)$. Hence $s_3s_4n=\sum\limits_{i=1}^ms_4r_iw_i+\sum\limits_{i=1}^lr'_itv_i.$ Thus $s_3s_4N\subseteq \langle w_1,\dots, w_m, tv_1,\dots, tv_l \rangle $ implies that $N$ is $S$-finite, which is a contradiction.

Let $F=\langle m_1, \dots , m_k\rangle$ be a submodule of $M$ such that $sM\subseteq F$ for some $s\in S$. Claim that $\p\cap S=\emptyset$. Indeed, if $s'\in \p$ for some $s'\in S$, then $s'M\subseteq N\subseteq M$. So $ss'N\subseteq ss'M\subseteq s'F \subseteq s'M \subseteq  N$ implies that $N$ is $S$-finite, which is a  contradiction. Note that  $\p=(N:M)\subseteq (N:F)\subseteq(N:sM)=(\p:s)=\p$ as $\p$ is a prime ideal of $R$. So $\p=(N:F)=(N:\langle m_1, \dots , m_k\rangle)=\bigcap\limits_{i=1}^k(N:Rm_i)$. By \cite[Proposition 1.11]{am69}, $\p=(N:Rm_j)$ for some $1\leq j\leq k$. Since $m_j\not\in N$, it follows that $N+Rm_j$ is $S$-finite. Let $\{y_1,\dots,y_m\}$ be a subset of $N+Rm_j$ such that $s_5(N+Rm_j)\subseteq \langle y_1,\dots,y_m \rangle$ for some $s_5\in S$. Write $y_i=w_i+a_im_j$ for some $w_i\in N$ and $a_i\in R ~(i=1,\dots,m)$. Let $n \in N$. Then $s_5n=\sum\limits_{i=1}^mr_i(w_i+a_im_j)=\sum\limits_{i=1}^mr_iw_i+(
\sum\limits_{i=1}^mr_ia_i)m_j$. Thus $(\sum\limits_{i=1}^mr_ia_i)m_j\in N $. So $\sum\limits_{i=1}^mr_ia_i\in \p$. Thus $s_5N\subseteq \langle w_1,\dots,  w_m \rangle +\p m_j$. As $\Ann(M)\subseteq (N:M)=\p$, there exists an $S$-finite submodule $N^\p$ of $M$ such that $\p M\subseteq N^\p\subseteq M(\p)$. Thus
\begin{eqnarray*}
  s_5N & \subseteq & \langle w_1,\dots, w_m \rangle + \p m_j \\
   & \subseteq & \langle w_1,\dots, w_m \rangle + \p M \\
   & \subseteq & \langle w_1,\dots,  w_m \rangle + N^\p \\
   &\subseteq & \langle w_1,\dots,  w_m \rangle +M(\p) \\
  & \subseteq & N
\end{eqnarray*}
Since $N^\p+\langle w_1,\dots,  w_m \rangle $ is $S$-finite, it follows that $N$ is also  $S$-finite,  which is a  contradiction.
Hence $M$ is $S$-Noetherian.
\end{proof}

Taking $S=\{1\}$, we can recover the following Parkash and Kour's result.

\begin{corollary} \cite[Theorem 2.1]{pk21} Let $R$ be a ring and $M$  a finitely generated $R$-module. Then $M$ is Noetherian if and only if for every prime ideal $\p$ of $R$ with $\Ann(M)\subseteq \p$, there exists a finitely generated submodule $N^\p$ of $M$ such that $\p M\subseteq N^\p\subseteq M(\p)$.
\end{corollary}

\section{The Cohen's theorem for $w$-Noetherian modules}
We recall some basic knowledge on the $w$-operation over a commutative ring. One can refer to {\cite{fk16}} for more details. Let $R$ be a commutative ring and $J$ a finitely generated ideal of $R$. Then $J$ is called a \emph{$\GV$-ideal} if the natural homomorphism $R\rightarrow \Hom_R(J,R)$ is an isomorphism. The set of $\GV$-ideals is denoted by $\GV(R)$. Let $M$ be an $R$-module. Define
$$\tor(M):=\{x\in M \mid Jx=0~\mbox{for some}~ J\in \GV(R) \}.$$
An $R$-module $M$ is said to be \emph{$\GV$-torsion} (resp., \emph{$\GV$-torsion-free}) if $\tor(M)=M$ (resp., $\tor(M)=0$). A $\GV$-torsion-free module $M$ is called a {$w$-module} if $\Ext_R^1(R/J,M)=0$ for any $J\in \GV(R)$. A \emph{$\DW$ ring} $R$ is a ring for which every $R$-module is a $w$-module.
A \emph{maximal $w$-ideal} is an ideal of $R$ which is maximal among the $w$-submodules of $R$. The set of all maximal $w$-ideals is denoted by $w$-$\Max(R)$. Each maximal $w$-ideals is a  prime ideal (see {\cite[Theorem 6.2.14]{fk16}}).

An $R$-homomorphism $f:M\rightarrow N$ is said to be a \emph{$w$-monomorphism} (resp., \emph{$w$-epimorphism}, \emph{$w$-isomorphism}) if for any $ \p\in w$-$\Max(R)$, $f_{\p}:M_{\p}\rightarrow N_{\p}$ is a monomorphism (resp., an epimorphism, an isomorphism). Note that $f$ is a $w$-monomorphism (resp., $w$-epimorphism) if and only if $\Ker(f)$ (resp., $\Coker(f)$) is $\GV$-torsion.
An $R$-module $M$ is said to be  \emph{$w$-finite type} if there exist a finitely generated free module $F$ and a $w$-epimorphism $g: F\rightarrow M$. Obviously, an $R$-module $M$ is $w$-finite type if and only if there is a finitely generated submodule $N$ of $M$ such that $M/N$ is $\GV$-torsion.

\begin{lemma}\label{loc-w-f} Let $N$ be a $w$-submodule of a $\GV$-torsion-free $w$-finite type  module $M$. Then $(N:_RM)_{\p}= (N_{\p}:_{R_{\p}}M_{\p})$ for any prime $w$-ideal $\p$ of $R$.
\end{lemma}

\begin{proof}
Let $\p$ be a prime $w$-ideal of $R$. Obviously, $(N:_RM)_{\p}\subseteq (N_{\p}:_{R_{\p}}M_{\p})$. On the other hand,  since $M$ is a $w$-finite type $R$-module, there exists a finitely generated submodule $F=\langle m_1,\dots,m_n\rangle$ of $M$ satisfying that  for any $m\in M$ there exists $J\in\GV(R)$  such that $Jm\subseteq F$. Let $\frac{r}{s}$  be an element in $(N_{\p}:_{R_{\p}}M_{\p})$.  Then for each $i=1,\dots,n$,  there exists $s_i\in R \setminus \p$ such that $s_irm_i\in N$. Thus $s_1\cdots s_n rF\in N$. So $s_1\cdots s_n rJm\in N$ for all $m\in M\subseteq E(M)$, where $E(M)$ is the injective envelope of $M$. By \cite[Theorem 6.16]{fk16}, $s_1\cdots s_nrM\subseteq N$  since $N$ is a $w$-module. Hence $s_1\cdots s_nr\in (N:_RM)$. Consequently,   $\frac{r}{s}=\frac{s_1\cdots s_nr}{s_1\cdots s_n s}\in (N:_RM)_{\p}$.
\end{proof}

Let $M$ be an $R$-module. Recall from \cite[Definition 8.1]{fk16} that $M$ is called a \emph{$w$-Noetherian module} if every submodule of $M$ is $w$-finite type. And  $R$ is called a \emph{$w$-Noetherian ring} if $R$ is $w$-Noetherian as an $R$-module.

\begin{theorem} Let $R$ be a ring and $M$ a  $\GV$-torsion-free $w$-finite type $R$-module. Then $M$ is a $w$-Noetherian module if and only if for every prime $w$-ideal $\p$ of $R$ with $\Ann(M)\subseteq \p$, there exists a $w$-finite type submodule $N^{\p}$ of $M$ such that $\p M\subseteq N^{\p}\subseteq M(\p)$.
\end{theorem}

\begin{proof} Suppose that $M$ is a  $w$-Noetherian $R$-module and let $\p$ be a prime $w$-ideal with $\Ann(M)\subseteq \p$.  If we  take $N^\p:=\p M$, then $N^\p$ is certainly a  $w$-finite type submodule of $M$ satisfying $\p M\subseteq N^\p\subseteq M(\p)$.

Conversely, suppose on the contrary that $M$ is not $w$-Noetherian. Let $\N$ be the set of all $w$-submodules of $M$ which are not $w$-finite type. Then $\N$ is non-empty. Make a partial order on $\N$ by defining  $N_1\leq N_2$ if and only if $N_1\subseteq N_2$ in $\N$. Let $\{N_i\mid i\in \Lambda\}$  be a chain in  $\N$. Set $N:=\bigcup\limits_{i\in \Lambda}N_i$. Then $N$ is not $w$-finite type. Indeed, suppose there is an exact $0\rightarrow F\rightarrow N\rightarrow T\rightarrow 0$ with $T$  $\GV$-torsion and $F=\langle x_1, \dots , x_n\rangle$ finitely generated. Then there exists $i_0\in \Lambda$ such that $F\subseteq N_{i_0}$.  Consider the following commutative diagram with exact rows:
$$\xymatrix@R=20pt@C=25pt{
0\ar[r] &F\ar@{=}[d]\ar[r]^{} &N_{i_0}\ar@{^{(}->}[d]\ar[r] & T'\ar@{^{(}->}[d]\ar[r]  & 0 \\
0 \ar[r]^{} &F  \ar[r]^{} & N\ar[r]^{} &T\ar[r]^{} & 0 \\ }$$
Since $T'$ is a submodule of $T$, we have that $T'$ being $\GV$-torsion implies that $N_{i_0}$ is $w$-finite type, which is a contradiction.  Since $N$ is a $w$-submodule of $M$, it follows that $N\in\N$. So by Zorn's Lemma, $\N$ has a maximal element, which is also denoted by $N$. Set $\p:=(N:M)=\{r\in R \mid rM\subseteq N\}$. Then $\p$ is a $w$-ideal by \cite[Section 6.10, Exercise 6.8]{fk16}.

We claim that $\p$ is a prime ideal of $R$. Assume on the contrary that there exist $a,b\in R \setminus \p$ such that $ab\in \p$. Since $a,b\in R \setminus \p$, we have $aM\not\subseteq N$ and $bM\not\subseteq N$. Therefore $N+aM$ is $w$-finite type. Let $\{y_1,\dots,y_m\}$ be a subset of $N+aM$ such that  $0\rightarrow F_1\rightarrow N+aM\rightarrow T_1\rightarrow 0$ be an exact sequence with $T_1$  $\GV$-torsion and $F_1=\langle y_1,\dots,y_m\rangle$ finitely generated. Write $y_i=w_i+az_i$ for some $w_i\in N$ and $z_i\in M\ (1\leq i\leq m)$. Set $L:=\{x\in M\mid ax\in N\}$. Then $N+bM\subseteq L$, and hence $L$ is $w$-finite type. Let $0\rightarrow F_2\rightarrow L\rightarrow T_2\rightarrow 0$ be an exact sequence with $T_2$  $\GV$-torsion and $F_2=\langle x_1,\dots,x_k\rangle$ finitely generated.
Let $n$ be an element in $N$. Then there is a $\GV$-ideal $J_1=\langle j^1_1,\dots,j^p_1\rangle$ such that $J_1n\subseteq F_1$. So there is $\{r^t_i\mid t=1,\dots,p; i=1,\dots, m \}\subseteq  R$ such that
$$j^t_1n=\sum\limits_{i=1}^mr^t_iy_i=\sum\limits_{i=1}^mr^t_iw_i+a\sum\limits_{i=1}^mr^t_iz_i\  (t=1,\dots,p).$$
 Then $\sum\limits_{i=1}^mr^t_iz_i\in L ~(t=1,\dots,p)$. Thus there exists a $\GV$-ideal $J_2=\langle j^1_2,\dots,j^l_2\rangle$ such that $j^s_2\sum\limits_{i=1}^mr^t_iz_i=\sum\limits_{i=1}^kr'^{t,s}_ix_i $ for some $\{r'^{t,s}_i\mid i=1,\dots,k; t=1,\dots,p; s=1,\dots,l\}\subseteq R$.
So $j^t_1j^s_2n=\sum\limits_{i=1}^mj^s_2r^t_iw_i+\sum\limits_{i=1}^kr'^{t,s}_iax_i\ (t=1,\dots,k; s=1,\dots,l)$. Thus $$J_1J_2n\subseteq \langle w_1,\dots, w_m, ax_1,\dots, ax_k \rangle$$ implies that $N$ is $w$-finite type, which is a  contradiction.

We claim that $M(\p)\subseteq N$. Assume on the contrary that there exists an element $y\in M(\p)$ such that $y\not\in N$. Then there exists $t'\in R \setminus \p$ such that $t'y\in \p M=(N:M)M\subseteq N$. As $t'\not\in \p=(N:M)$,  it follows that $t'M\not\subseteq N$. Therefore $N+t'M$ is $w$-finite type.
Let $0\rightarrow F_3\rightarrow N+t'M\rightarrow T_3\rightarrow 0$ be an exact sequence with $T_3$  $\GV$-torsion and $F_3=\langle u_1,\dots,u_m\rangle$ a finitely generated submodule of $N+t'M$.
Write $u_i=w_i+t'z_i ~(i=1,\dots,m)$ with $w_i\in N$ and $z_i\in M$. Set $L:=\{x\in M \mid tx\in N\}$. Then $N\subset N+Ry\subseteq L$, and hence $L$ is $w$-finite type. Let $0\rightarrow F_4\rightarrow L\rightarrow T_4\rightarrow 0$ be an exact sequence with $T_4$  $\GV$-torsion and $F_4=\langle u_1,\dots,u_n\rangle$ a finitely generated submodule of $L$. Let $n$ be an element in $N$. Then there is a $\GV$-ideal $J_3=\langle j^1_3,\dots,j^k_3\rangle$ such that $J_3n\subseteq F_3$. So there is $\{r^t_i\mid t=1,\dots,p; i=1,\dots, m \}\subseteq  R$ such that
 $$j^t_3n=\sum\limits_{i=1}^mr^t_iu_i=\sum\limits_{i=1}^mr^t_iw_i+t'\sum\limits_{i=1}^mr^t_iz_i\  (t=1,\dots,p).$$
So $\sum\limits_{i=1}^mr^t_iz_i\in L\  (t=1,\dots,p)$. Thus there exists a $\GV$-ideal $J_4=\langle j^1_4,\dots,j^l_4\rangle$ such that $j^s_4\sum\limits_{i=1}^mr^t_iz_i=\sum\limits_{i=1}^nr'^{t,s}_iu_i$ for some $\{r'^{t,s}_i\mid i=1,\dots,m; t=1,\dots,p; s=1,\dots,l\}\subseteq R$. So $j^t_3j^s_4n=\sum\limits_{i=1}^mj^s_4r^t_iw_i+\sum\limits_{i=1}^kr'^{t,s}_it'u_i\ (t=1,\dots,k; s=1,\dots,l)$. Thus $J_3J_4n\subseteq \langle w_1,\dots, w_m, t'u_1,\dots, t'u_k \rangle $ implies that $N$ is $w$-finite type, which is a  contradiction.

Let $\m$ be a maximal $w$-ideal of $R$ and $F=\langle m_1, \dots , m_k\rangle$ a submodule of $M$ such that $M/F$ is $\GV$-torsion. So $M_\m=F_\m$. Then $(N:_RM)_\m=(N_\m:_{R_\m}M_\m)=(N_\m:_{R_\m}F_\m)=(N:_RF)_\m$ by Lemma \ref{loc-w-f}. By \cite[Section 6.10, Exercise 6.8]{fk16}, $(N:_RM)$ and $(N:_RF)$ are all $w$-ideals. So we have $\p=(N:_RM)=(N:_RF)=\bigcap\limits_{i=1}^k(N:Rm_i)$. By \cite[Proposition 1.11]{am69}, $\p=(N:_RRm_j)$ for some $1\leq j\leq k$. Since $m_j\not\in N$, it follows that $N+Rm_j$ is $w$-finite type.  Let $0\rightarrow F_5\rightarrow N+Rm_j\rightarrow T_5\rightarrow 0$ be an exact sequence with $T_5$  $\GV$-torsion and $F_5=\langle y_1,\dots,y_m\rangle$ a finitely generated submodule of $N+Rm_j$.  Write $y_i=w_i+a_im_j$ for some $w_i\in N$ and $a_i\in R~ (i=1,\dots,m)$. Let $n$ be an element in $N$.  Then there is a $\GV$-ideal $J_5=\langle j^1_5,\dots,j^l_5\rangle$ such that $J_5n\subseteq F_5$. So there is $\{r^t_i\mid t=1,\dots,p; i=1,\dots, m \}\subseteq  R$ such that  $j_5^tn=\sum\limits_{i=1}^mr_i^ty_i=\sum\limits_{i=1}^mr_i^tw_i+(
\sum\limits_{i=1}^mr_i^ta_i)m_j\ (t=1,\dots,l)$. So $\sum\limits_{i=1}^mr_i^ta_i\in \p$.  Thus $J_5N\subseteq \langle w_1,\dots,  w_m \rangle +\p m_j$. As $\Ann(M)\subseteq (N:M)=\p$, there exists a $w$-finite type submodule $N^{\p}$ of $M$ such that $\p M\subseteq N^{\p}\subseteq M(\p)$. Thus
\begin{eqnarray*}
  J_5N & \subseteq & \langle w_1,\dots, w_m \rangle + \p m_j \\
   & \subseteq & \langle w_1,\dots, w_m \rangle + \p M \\
   & \subseteq & \langle w_1,\dots,  w_m \rangle + N^\p \\
   &\subseteq & \langle w_1,\dots,  w_m \rangle +M(\p) \\
  & \subseteq & N
\end{eqnarray*}
Since $N^{\p}+\langle w_1,\dots,  w_m \rangle $ is $w$-finite type, it follows that $N$ is also $w$-finite type, which is a  contradiction.
Hence $M$ is $w$-Noetherian.
\end{proof}

Taking $M:=R$, we have the following characterization of $w$-Noetherian rings.

\begin{corollary} \cite[Theorem 4.7(1)]{ywzc11} Let $R$ be a ring. Then $R$ is a $w$-Noetherian ring if and only if each prime $w$-ideal of $R$ is $w$-finite type.
\end{corollary}

\begin{acknowledgement}\quad\\
The first author was supported by the National Natural Science Foundation of China (No. 12061001).
\end{acknowledgement}


\begin{thebibliography}{99}

\bibitem{ad02}  D. D.  Anderson and T.  Dumitrescu,  $S$-Noetherian rings, {\it Comm. Algebra} {\bf 30} (2002), 4407-4416.

\bibitem{am69}  M. F. Atiyah, I. G. MacDonald, {\it Introduction to Commutative Algebra}, (Addison-Wesley, 1969).

\bibitem{CHK}  W. Chung, J. Ha, and H. Kim, Some remarks on strong Mori domains, {\it Houston J. Math.} 38 (2012), no. 4, 1051-1059.

\bibitem{c50} I. S. Cohen,  Commutative rings with restricted minimum condition, {\it Duke Math. J.} 17 (1950), 27-42.

\bibitem{J00} P. Jothilingam, Cohen's theorem and Eakin-Nagata theorem revisited, {\it Comm. Algebra}  28  (2000), no. 10,  4861-4866.

\bibitem{kkl14} H. Kim,  M. O. Kim, and J. W.  Lim,    On $S$-strong Mori domains, {\it J. Algebra}   416  (2014),  314-332.

\bibitem{kmz21} H. Kim, N. Mahdou, and Y. Zahir,  $S$-Noetherian in Bi-amalgamations, {\it Bull. Korean Math. Soc.}  (to appear).

\bibitem{l15} J. W. Lim, A note on $S$-Noetherian domains, {\it Kyungpook Math. J.}  55 (2015), 507-514.

\bibitem{l07} Z. Liu,  On $S$-Noetherian rings, {\it Arch. Math.} (Brno)  43 (2007), 55-60.

\bibitem{mccs93} R. L. McCasland and P. F. Smith, Prime submodules of Noetherian modules, {\it Rocky Mountain J. Math.} 23 (1993), no. 3, 1041-1062.

\bibitem{pk21} A. Parkash and S. Kour, On Cohen's theorem for modules, {\it Indian J. Pure Appl. Math.} https://doi.org/10.1007/s13226-021-00101-z.

\bibitem{s94}   P. F.  Smith, Concerning a theorem of I. S. Cohen. XIth National Conference of Algebra (Constanta, 1994), An. Stiint. Univ.
Ovidius Constanta Ser. Mat. 2 (1994), 160-167.

\bibitem{fk16} F. G.  Wang and  H. Kim,  {\it  Foundations of Commutative Rings and Their Modules} (Singapore, Springer, 2016).

\bibitem{fm97} F. Wang and R. L. McCasland, On $w$-modules over strong Mori domains,  {\it  Comm. Algebra} 25 (1997), 1285-1306.

\bibitem{ywzc11}   H. Yin,  F. G. Wang, X. Zhu, and Y. Chen, $w$-modules over commutative rings, {\it  J. Korean Math. Soc.} 48 (2011), no. 1, 207-222.

\bibitem{zwk} J. Zhang, F. G. Wang, and H.  Kim,  Injective modules over $w$-Noetherian rings, II. {\it J. Korean Math. Soc.}  50 (2013), no. 6, 1051-1066.
\end{thebibliography}
\end{document}